\documentclass[12pt]{amsart}

\usepackage{amscd,amssymb,amsopn,amsmath,amsthm,mathrsfs,graphics,amsfonts,enumerate,verbatim,calc
}

\usepackage{color}

\usepackage[OT2,OT1]{fontenc}
\newcommand\cyr{%
\renewcommand\rmdefault{wncyr}%
\renewcommand\sfdefault{wncyss}%
\renewcommand\encodingdefault{OT2}%
\normalfont
\selectfont}
\DeclareTextFontCommand{\textcyr}{\cyr}

\usepackage{amssymb,amsmath}
\usepackage{caption}

\DeclareFontFamily{OT1}{rsfs}{}
\DeclareFontShape{OT1}{rsfs}{n}{it}{<-> rsfs10}{}
\DeclareMathAlphabet{\mathscr}{OT1}{rsfs}{n}{it}

\numberwithin{equation}{section}
\newtheorem{theorem}{Theorem}[section]
\newtheorem{lemma}[theorem]{Lemma}
\newtheorem{proposition}[theorem]{Proposition}

\newtheorem{question}{Question}

\theoremstyle{definition}
\newtheorem{definition}[theorem]{Definition}
\newtheorem{remark}[theorem]{Remark}
\theoremstyle{remark}

\newtheorem{acknowledgement}{Acknowledgement}

\newcommand{\Ass}{\operatorname{Ass}}

\newcommand{\im}{\operatorname{Im}}
\renewcommand{\ker}{\operatorname{Ker}}

\newcommand{\Spec}{\operatorname{Spec}}

\newcommand{\Ann}{\operatorname{Ann}}

\newcommand{\fm}{\frak{m}}
\newcommand{\fp}{\frak{p}}
\newcommand{\fq}{\frak{q}}

\begin{document}
\title[Frobenius test exponents]{Frobenius test exponent for ideals generated by filter regular sequences}

\author[Duong Thi Huong]{Duong Thi Huong}
\address{Department of Mathematics, Thang Long University, Hanoi, Vietnam}
\email{huongdt@thanglong.edu.vn}

\author[Pham Hung Quy]{Pham Hung Quy}
\address{Department of Mathematics, FPT University, Hanoi, Vietnam}
\email{quyph@fe.edu.vn}

\thanks{
The authors are partially supported by a fund of Vietnam National Foundation for Science
and Technology Development (NAFOSTED) under grant number
101.04-2020.10.}


\maketitle
\begin{center}
{\textit{Dedicated to Professor Nguyen Tu Cuong on the occasion of his 70th birthday}}
\end{center}

\begin{abstract} Let $(R,\fm)$ be a Noetherian local ring of prime characteristic $p>0$, and $t$ an integer such that $H_{\fm}^j(R)/0^F_{H^j_{\fm}(R)}$ has finite length for all $j<t$. The aim of this paper is to show that there exists an uniform bound for Frobenius test exponents of ideals generated by filter regular sequences of length at most $t$. 
\end{abstract}

\section{Introduction}
Throughout this paper, let $R$ be a Noetherian commutative ring of prime characteristic $p>0$, and $I$ an ideal of $R$. The key ingredient in study ring of prime characteristic is using the Frobenius endomorphism $F: R \to R; x \mapsto x^p$ and its $e$-th iterations $F^e$, $e \ge 1$. The {\it $e$-th Frobenius power} of $I$  is the extension of $I$ via $F^e$, $I^{[p^e]} = (x^{p^e} \mid x \in I)$. The {\it Frobenius closure} of $I$ is $I^F = \{x \mid x^{p^e} \in I^{[p^e]} \text{ for some } e \ge 0\}$. Hence $I^F$ is the set of all nilpotent elements under Frobenius endomorphism modulo $I$. By the Noetherianess of $R$ there is an integer $e$, depending on $I$, such that $(I^F)^{[p^e]} = I^{[p^e]}$. The {\it Frobenius test exponent} of $I$ is defined by $\mathrm{Fte}(I) = \min \{e \mid (I^F)^{[p^e]} = I^{[p^e]}\}$. Under mild conditions, $R$ is $F$-pure if and only if $\mathrm{Fte}(I) = 0$ for all $I$ by Hochster \cite{H77}. In general, we can not expect an upper bound for all $\mathrm{Fte}(I)$ by Brenner \cite{B06}. Restricting to the class of parameter ideals, for any local ring $(R, \fm)$ of dimension $d$ we define the {\it Frobenius test exponent for parameter ideals} $\mathrm{Fte}(R)$ as follows
$$\mathrm{Fte}(R) = \min \{e \mid (\fq^F)^{[p^e]} = \fq^{[p^e]} \, \text{for all parameter ideals }\, \fq\},$$
 and $\mathrm{Fte}(R) = \infty$ if we have no such integer. Katzman and Sharp asked whether $\mathrm{Fte}(R) < \infty$ for any (equidimensional) local ring. Moreover they \cite{KS06} proved that it is the case when the ring is Cohen-Macaulay.  The main idea in \cite{KS06} is connecting $\mathrm{Fte}(R)$ with an invariant defined by the Frobenius actions on the local cohomology modules $H^i_{\fm}(R)$, namely the {\it Hartshorne-Speiser-Lyubeznik number} of $H^i_{\fm}(R)$.  Recall that the Frobenius endomorphism induces natural Frobenius actions on the local cohomology modules $H^i_{\fm}(R)$ for all $i \ge 0$.  The {\it Frobenius closure of zero submodule} of $H^i_{\fm}(R)$, denoted by $0^F_{H^i_{\fm}(R)}$, consists nilpotent elements of $H^i_{\fm}(R)$ under the Frobenius action. The Hartshorne-Speiser-Lyubeznik number of $H^i_{\fm}(R)$ is defined by $$\mathrm{HSL}(H^i_{\fm}(R)) = \min \{e \mid 0^F_{H^i_{\fm}(R)} = \mathrm{Ker}(H^i_{\fm}(R) \overset{F^e}{\longrightarrow} H^i_{\fm}(R))\}.$$ 
It is proved that $\mathrm{HSL}(H^i_{\fm}(R))$ always exists. The Hartshorne-Speiser-Lyubeznik number of $R$ is $\mathrm{HSL}(R) = \max \{\mathrm{HSL}(H^i_{\fm}(R)) \mid i = 0, \ldots, d\}$. In fact Katzman and Sharp proved that $\mathrm{Fte}(R) = \mathrm{HSL} (R)$ provided $R$ is Cohen-Macaulay. In general, we have $\mathrm{Fte}(R)\geq \mathrm{HSL}(R)$ for any local ring by \cite{HQ19}. Beyond the Cohen-Macaulay case, Huneke, Katzman, Sharp, and Yao \cite{HKSY06} showed that $\mathrm{Fte}(R) < \infty$ for generalized Cohen-Macaulay local rings by using several concepts and techniques from commutative algebra, namely unconditioned strong d-sequences, cohomological annihilators, and modules of generalized fractions. In 2019, the second author \cite{Q19} not only simplified the proof for generalized Cohen-Macaulay rings but also proved $\mathrm{Fte}(R)<\infty$ for {\it weakly F-nilpotent} rings, i.e. $H^i_{\fm}(R) = 0^F_{H^i_{\fm}(R)}$ for all $i<d$. Recently, Maddox \cite{M19} extended this result for {\it generalized weakly F-nilpotent} rings, i.e. $H^i_{\fm}(R)/0^F_{H^i_{\fm}(R)}$ has finite length for all $i<d$.
 
Suppose $\mathrm{Fte}(R) < \infty$, and let $x_1, \ldots, x_d$ be a system of parameters of $R$. Then for any $t\le d$ we have $\mathrm{Fte}((x_1, \ldots, x_t)) \le \mathrm{Fte} (R)$ by using Krull's intersection theorem. If we do not know about the finiteness of $\mathrm{Fte}(R)$, in which conditions we have $\mathrm{Fte}((x_1, \ldots, x_t)) < \infty$? For a regular sequence $\underline{x}=x_1,\ldots ,x_t$, Katzman and Sharp \cite[Corollary 4.3]{KS06} showed that there exists an integer $C_{\underline{x}}$ such that $\mathrm{Fte}((x_1^{n_1},\ldots,x_t^{n_t}))\le C_{\underline{x}}$ for all $n_1,\ldots,n_t \ge 1$. Since $C_{\underline{x}}$ depends on the choice of $\underline{x}$, the previous result does not answer for the following question.
\begin{question} Let $(R, \fm)$ be a local ring of prime characteristic and of dimension $d$, and $t \le \mathrm{depth}(R)$ an integer. Does there exist a positive integer $C$ such that for any regular sequence $x_1, \ldots, x_t$ we have $\mathrm{Fte} ((x_1, \ldots, x_t)) \le C$?
\end{question}
Inspired by known results about the finiteness of $\mathrm{Fte}(R)$, we give an affirmative answer for the above question in a more general form. It covers all aforementioned results.
\begin{theorem}\label{main thm}
Let $(R, \fm)$ be a Noetherian local ring of dimension $d$ and of prime characteristic $p>0$, and $t \le d$ a positive integer such that $H_{\fm}^j(R)/0^F_{H^j_{\fm}(R)}$ has finite length for all $j<t$. Then there exists a positive integer $C$ such that for any filter regular sequence $x_1,\ldots,x_t$ we have $\mathrm{Fte}((x_1, \ldots, x_t)) \leq C$.
\end{theorem}
In the next section we recall the basic notions and relevant materials. We prove the main theorem in the last section.
\section{Preliminaries}
\subsection{Filter regular sequences}
In this subsection, the assumption that $R$ is of prime characteristic is unnecessary.
\begin{definition}
Let $(R,\fm)$ be a Noetherian local ring, and $x_1,\ldots, x_t$ a sequence of elements of $R$. Then we say that $x_1,\ldots, x_t$ is a {\it filter regular sequence} if the following conditions hold:
\begin{enumerate}
\item $(x_1,\ldots, x_t)\subseteq \fm$,
\item $x_i\notin \fp$ for all $\fp\in \Ass \big(\frac{M}{(x_1,\ldots,x_{i-1})M}\big)\setminus \{\fm\}$, $i=1,\ldots,t$.
\end{enumerate}
\end{definition}
The notion of filter regular sequences was introduced by Cuong, Schenzel, and Trung in \cite{CST78}. The following is well-known. 
\begin{lemma}
Let $(R,\fm)$ be a Noetherian local ring and $x_1,\ldots, x_t\subseteq \fm$. The following are equivalent
\begin{enumerate}
\item $x_1,\ldots,x_t$ is a filter regular sequence.
\item For each $1\le i\le t$ the quotient
$$\frac{(x_1,\ldots,x_{i-1}):_Rx_i}{(x_1,\ldots,x_{i-1})}$$
is an $R$-module of finite length.
\item For each $1\le i\le t$ the sequence
$$\frac{x_1}{1},\frac{x_2}{1},\ldots,\frac{x_i}{1}$$
forms an $R_{\fp}$-regular sequence for every $\fp\in \Spec (R/(x_1,\ldots,x_i)) \setminus \{\fm\}$.
\item The sequence $x_1^{n_1},\ldots,x_t^{n_t}$ is a filter regular sequence for all $n_1,\ldots,n_t\geq 1$.
\item The sequence $x_1,\ldots, x_t$ is a filter regular sequence of $\widehat{R}$, where $\widehat{R}$ is the $\fm$-adic completion of $R$.
\end{enumerate} 
\end{lemma}

\subsection{The Frobenius action on local cohomology}
In this subsection, let $R$ be a Noetherian ring containing a field of characteristic $p>0$. Let $F:R \to R, x \mapsto x^p$ denote the Frobenius endomorphism. If we want to notationally distinguish the source and target of the $e$-th Frobenius endomorphism $F^e: R \xrightarrow{x \mapsto x^{p^e}} R$, we will use $F_*^e(R)$ to denote the target. $F_*^e(R)$ is an $R$-bimodule, which is the same as $R$ as
an abelian group and as a right $R$-module, that acquires its left $R$-module structure via the $e$-th Frobenius
endomorphism $F^e$. By definition the $e$-th Frobenius endomorphism $F^e: R \to F_*^e(R)$ sending $x$ to $F_*^e(x^{p^e}) = x \cdot F_*^e(1)$ is an $R$-homomorphism. We say $R$ is {\it $F$-finite} if $F_*(R)$ is a finite $R$-module.
\begin{definition} Let $I$ be an ideal of $R$, we define
\begin{enumerate}
\item The {\it $e$-th Frobenius power} of $I$ is $I^{[p^e]} = (x^{p^e} \mid x \in I)$.
  \item The {\it Frobenius closure} of $I$, $I^F = \{x \mid  x^{p^e} \in I^{[p^e]} \text{ for some } e \ge 0\}$.
\end{enumerate}
\end{definition}
\begin{definition}
Let $I$ be an ideal of $R$. By the Noetherianess of $R$ there is an integer $e$, depending on $I$ such that $(I^F)^{[p^e]}=I^{[p^e]}$. The smallest number $e$ satisfying the condition is called the {\it Frobenius test exponent of $I$}, and denoted by $\mathrm{Fte}(I)$,
$$\mathrm{Fte}(I)=\min \{e\mid (I^F)^{[p^e]}=I^{[p^e]}\}.$$
\end{definition}
A problem of Katzman and Sharp \cite[Introduction]{KS06} asks in its strongest form: does there exist an integer $e$, depending only on the ring $R$, such that for every ideal $I$ we have $(I^F)^{[p^e]} = I^{[p^e]}$. A positive answer to
this question, together with the actual knowledge of a bound for $e$, would give an algorithm to compute the Frobenius closure $I^F$. Unfortunately, Brenner \cite{B06} gave two-dimensional normal standard graded domains with no Frobenius test exponent. In contrast, Katzman and Sharp showed the existence of Frobenius test exponent if we restrict to class of parameter ideals in a Cohen-Macaulay ring. Therefore it is natural to ask the following question.
\begin{question}\label{Katzman Sharp} Let $(R, \fm)$ be an (equidimensional) local ring of prime characteristic $p$. Then does there exist an integer $e$ such that for every parameter ideal $\fq$ of $R$ we have $(\fq^F)^{[p^e]} = \fq^{[p^e]}$?
\end{question}
We define the {\it Frobenius test exponent for parameter ideals} of $R$, $\mathrm{Fte}(R)$, the smallest integer $e$ satisfying the above condition and $\mathrm{Fte}(R) = \infty$ if we have no such $e$. It should be noted that the authors recently used the finiteness of $\mathrm{Fte}(R)$, if have, to find an upper bound of the multiplicity of a local ring \cite{HQ20}.\\
For any ideal $I = (x_1, \ldots, x_t)$, the Frobenius endomorphism $F:R \to R$ and its localizations induce a natural Frobenius action on local cohomology $F:H^i_I(R) \to H^i_{I^{[p]}}(R) \cong H^i_{I}(R)$ for all $i \ge 0$. In general, let $A$ be an Artinian $R$-module with a Frobenius action $F: A \to A$. Then we define the {\it Frobenius closure $0^F_A$ of the zero submodule} of $A$ is the submodule of $A$ consisting all elements $z$ such that $F^e(z) = 0$ for some $e \ge 0$. Hence $0^F_A$ is the nilpotent part of $A$ under the Frobenius action. By \cite[Proposition 1.11]{HS77}, \cite[Proposition 4.4]{L97} and \cite{Sh07} we have the following.
\begin{theorem}
Let $(R, \fm)$ be a local ring of prime characteristic $p>0$, and $A$ an Artinian $R$-module with a Frobenius action $F: A \to A$. Then there exists a non-negative integer $e$ such that $0^F_A = \ker (A \overset{F^e}{\longrightarrow} A)$.
\end{theorem}
\begin{definition}
\begin{enumerate}
\item Let $A$ be an Artinian $R$-module with a Frobenius action $F$.  The {\it Hartshorne-Speiser-Lyubeznik number} of $A$ is denoted by $\mathrm{HLS}(A)$ and is defined to be
$$\mathrm{HSL}(A)=\min\{e\mid 0^F_A = \ker (A \overset{F^e}{\longrightarrow} A)\}.$$
\item Notice that $H^i_{\fm}(R)$ is always Artinian for all $i \ge 0$. We define the {\it Hartshorne-Speiser-Lyubeznik number} of a local ring $(R, \frak m)$ as follows
$$\mathrm{HSL}(R): = \min \{ e \mid   0^F_{H^i_{\fm}(R)} =   \ker (H^i_{\fm}(R) \overset{F^e}{\longrightarrow} H^i_{\fm}(R)) \text{ for all } i = 0, \ldots, d\}.$$
\end{enumerate}
\end{definition}
 As mentioned in the introduction, Question \ref{Katzman Sharp} was answered affirmatively in the following cases:
\begin{enumerate}
\item (Katzman-Sharp) $R$ is Cohen-Macaulay ring. Moreover, $\mathrm{Fte}(R)= \mathrm{HSL}(R)$.
\item (Huneke-Katzman-Sharp-Yao) $R$ is generalized Cohen-Macaulay ring, i.e. $H^i_{\fm}(R)$ has finite length for all $i < d$.
\item (Quy) $R$ is weakly $F$-nilpotent ring, i.e. $H^i_{\fm}(R) = 0^F_{H^i_{\fm}(R)}$ for all $i < d$.
\item (Maddox) $R$ is generalized weakly $F$-nilpotent ring, i.e. $H^i_{\fm}(R) / 0^F_{H^i_{\fm}(R)}$ has finite length for all $i < d$.
\end{enumerate}
The main idea of the proofs of the affirmative cases is to make a connection between the Frobenius test exponent for parameter ideals with Hartshorne-Speiser-Lyubeznik numbers of the local cohomology modules. For example, if $R$ is weakly $F$-nilpotent, then the second author \cite{Q19} used the notion of relative Frobenius action on local cohomology to prove
$$\mathrm{Fte}(R) \le \sum_{i = 0}^d \binom{d}{i} \mathrm{HSL} (H^i_{\fm}(R)).$$
\subsection{The relative Frobenius action on local cohomology}
In this subsection, we recall the notion of relative Frobenius action on local cohomology which was introduced in \cite{PQ19} by Polstra and Quy in study $F$-nilpotent rings. Let $K\subseteq I$ be ideals of $R$. The Frobenius endomorphism $F: R/K \to R/K$ can be factored as composition of two natural maps:
$$R/K\to R/K^{[p]}\twoheadrightarrow R/K,$$
where the second map is the natural project map. We denote the first map by $F_R:R/K\to R/K^{[p]}$, $F_R(a+K)=a^p+K^{[p]}$ for all $a\in R$. The homomorphism $F_R$ induces {\it the relative Frobenius actions on local cohomology} $F_R: H^i_I(R/K)\to H^i_I(R/K^{[p]})$ via \v{C}ech complexes.
\begin{definition}
\begin{enumerate}
\item We define the {\it relative Frobenius closure of the zero submodule of $H^i_I(R/K)$ with respect to $R$} as follows
$$0^{F_R}_{H^i_I(R/K)}=\{\eta \mid F^e_R(\eta)=0\in H^i_I(R/K^{[p^e]})\text{ for some }e\gg 0\}.$$
\item If there is an integer $e$ such that
$$0^{F_R}_{H^i_I(R/K)}=\ker (H^i_I(R/K)\xrightarrow{F^e_R} H^i_I(R/K^{[p^e]})),$$
then we call the smallest of such integers the {\it Hartshorne-Speiser-Lyubeznik number of $H^i_I(R/K)$ with respect to $R$}, denoted by $\mathrm{HSL}_R(H^i_I(R/K))$. And convention that $\mathrm{HSL}_R(H^i_I(R/K))=\infty$ if we have no such integer.
\end{enumerate}
\end{definition}
\begin{lemma}\label{Fro0} 
Let $(R, \fm)$ be a Noetherian local ring, and $I$ an ideal of $R$. Then 
$$0^{F_R}_{H_{\fm}^0(R/I)}=I^F\cap (I:\fm^{\infty})/I.$$
\end{lemma}
\begin{proof}
Pick $a+I\in 0^{F_R}_{H_{\fm}^0(R/I)}$. We have $a\in I:\fm^{\infty}$ and there exists an integer $e$ such that $F^e(a+I)=a^{p^e}+I^{[p^e]}=0\in R/I^{[p^e]}$. Equivalently, $a+I\in I^F\cap (I:\fm^{\infty})/I $.
\end{proof}

\section{Main result}
Throughout this section, let $(R,\fm)$ be a Noetherian local ring of prime characteristic $p>0$ and of dimension $d$, and $t \le d$ an integer such that $H_{\fm}^j(R)/0^F_{H^j_{\fm}(R)}$ has finite length for all $j<t$. Let $x_1,\ldots,x_t$ be a filter regular sequence. Set $I_i=(x_1,\ldots,x_i)$ for all $i\leq t$ and $I=(x_1,\ldots,x_t)$.
\begin{lemma}\label{GF-nilWeak} 
Let $(R,\fm)$ be a local ring of prime characteristic $p>0$ and of dimension $d$, and $t \le d$ an integer such that $H_{\fm}^j(R)/0^F_{H^j_{\fm}(R)}$ has finite length for all $j<t$. Let $n_0$ be an non-negative integer such that $\fm^{n_0} H_{\fm}^j(R)/0^F_{H^j_{\fm}(R)} = 0$ for all $j < t$. Then for every filter regular sequence $x_1,\ldots,x_t$ we have
$$\fm^{2^in_0} H^j_{\fm}(R/I_i)\subseteq 0^{F_R}_{H^j_{\fm}(R/I_i)}$$  for all $i \le t$ and all $j<t-i$.
\end{lemma}
\begin{proof}
We proceed by induction on $i$. The case $i=0$ is nothing to do.
Suppose $i>0$ and the assertion holds true for $i-1$. For each $e \ge 1$  we consider commutative diagram 
$$
\begin{CD}
0\longrightarrow R/(I_{i-1}:x_i) @>x_i>>  R/I_{i-1}  @>>> R/I_{i} @>>> 0 \\
@V(F^e_R)'VV @VF^e_RVV @VF^e_RVV (\star )\\
0\longrightarrow R/(I_{i-1}^{[p^e]}:x_i^{p^e}) @>x_i^{p^e}>> R/I_{i-1}^{[p^e]} @>>> R/I_{i}^{[p^e]} @>>> 0
\end{CD}
$$
where the first vertical map is the composition
$$R/(I_{i-1}:x_i)\xrightarrow{F^e_R}R/(I_{i-1}:x_i)^{[p^e]}\twoheadrightarrow R/(I_{i-1}^{[p^e]}:x_i^{p^e}). $$
Because $x_1^{p^e},\ldots,x_t^{p^e}$ is a filter regular sequence, $(I_{i-1}^{[p^e]}:x_i^{p^e})/I_{i-1}^{[p^e]}$ has finite length. Thus
$$H^{j+1}_{\fm}(R/(I_{i-1}^{[p^e]}:x_i^{p^e}))\cong H^{j+1}_{\fm}(R/(I_{i-1}:x_i)^{[p^e]})\cong H^{j+1}_{\fm}(R/I_{i-1}^{[p^e]}) $$
for all $j\geq 0$ and for all $e \ge 0$. Therefore the diagram $(\star)$ induces the following commutative diagram with exact rows
$$
\begin{CD}
\cdots\xrightarrow{x_i} H_{\fm}^j(R/I_{i-1}) @>>> H_{\fm}^j( R/I_{i})  @>\delta>>H_{\fm}^{j+1}( R/I_{i-1})  \\
@VF^e_R VV @VF^e_RVV @VF^e_RVV  \\
\cdots\xrightarrow{x_i^{p^e}} H_{\fm}^j(R/I_{i-1}^{[p^e]}) @>\alpha>> H_{\fm}^j(R/I_{i}^{[p^e]}) @>\beta>> H_{\fm}^{j+1}(R/I_{i-1}^{[p^e]}).
\end{CD}
$$
Pick any $u\in H^j_{\fm}(R/I_i)$, and $x,y\in \fm^{2^{i-1}n_0}$. Then $\delta(u)\in H^{j+1}_{\fm}(R/I_{i-1})$. By the inductive hypothesis $x\delta(u)\in 0^{F_R}_{H_{\fm}^{j+1}( R/I_{i-1})}$, so there is an integer $e$ such that 
$$0=F^e_R(x\delta(u))=\beta(F^e_R(xu)).$$
Hence there exists $v\in H_{\fm}^j(R/I_{i-1}^{[p^e]})$  such that $\alpha(v)=F^e_R(xu)$. Moreover, $y^{p^e}\in \fm^{2^{i-1}n_0}$ so $y^{p^e}v\in 0^{F_R}_{H_{\fm}^j(R/I_{i-1}^{[p^e]})}$ by using the induction for the sequence $x_1^{p^e}, \ldots, x_{i-1}^{p^e}$. Therefore,
$$\alpha(y^{p^e}v)=y^{p^e}.F^e_R(xu)=F^e_R(yxu)\in 0^{F_R}_{H_{\fm}^j(R/I_{i}^{[p^e]})}.$$
Leading to $xyu\in 0^{F_R}_{H_{\fm}^j( R/I_{i})}$. Hence $xy\in \Ann_R(H_{\fm}^j(R/I_i)/0^{F_R}_{H^j_{\fm}(R/I_i)})$, and so
$$\fm^{2^in_0}\subseteq \Ann_R(H_{\fm}^j(R/I_i)/0^{F_R}_{H^j_{\fm}(R/I_i)})$$ 
for all $i\leq t$ and $j<t-i$. The proof is complete.
\end{proof}

\begin{remark}\label{R0.10} 
Using the notation in Lemma \ref{GF-nilWeak}, and suppose $x_1 \ldots, x_t \in \fm^{2^tn_0}$. Then we have 
$$\im (H^j_{\fm}(R/I_{i-1}:x_i) \xrightarrow{x_i} H^j_{\fm}(R/I_{i-1}))\subseteq 0^{F_R}_{H^j_{\fm}(R/I_{i-1})}$$ for all $i \le t$ and $j \le t-i$. Indeed the case $1 \le j \le t-i$ follows from Lemma \ref{GF-nilWeak} and the fact $H^j_{\fm}(R/I_{i-1}:x_i) \cong H^j_{\fm}(R/I_{i-1})$. For $j=0$ we use Lemma \ref{GF-nilWeak} and the surjective map $H^0_{\fm}(R/I_{i-1}) \twoheadrightarrow H^0_{\fm}(R/I_{i-1}:x_i)$. In particular, every relative nilpotent element of the induced Frobenius action on $\mathrm{Coker} (H^j_{\fm}(R/I_{i-1}:x_i) \xrightarrow{x_i} H^j_{\fm}(R/I_{i-1}))$ is an image of some element in $0^{F_R}_{H^j_{\fm}(R/I_{i-1})}$.
\end{remark}
Using the above Remark and the argument of \cite[Proof of the main theorem]{Q19} and \cite[Theorem 3.1]{M19} we obtain the following whose proof is left to the reader.
\begin{proposition}\label{case dim zero}
Let $(R,\fm)$ be a local ring of prime characteristic $p>0$ and of dimension $d$, and $t \le d$ an integer such that $H_{\fm}^j(R)/0^F_{H^j_{\fm}(R)}$ has finite length for all $j<t$.  Then there exists a non-negative integer $e_0$ such that for every filter regular sequence $x_1,\ldots,x_t$ we have 
$$\mathrm{HSL}_R(H_{\fm}^0(R/(x_1, \ldots, x_t)))\leq\sum\limits_{k=0}^t\binom{t}{k}\mathrm{HSL}(H_{\fm}^k(R))+e_0.$$
\end{proposition}
\begin{remark}\label{regular case} If $x_1, \ldots, x_t$ is a regular sequence then $H^j_{\fm}(R) = 0$ for all $j<t$, and the number $e_0$ in the previous result can be chosen as zero. Thus we have $\mathrm{HSL}_R(H_{\fm}^0(R/(x_1, \ldots, x_t)))\le \mathrm{HSL}(H^t_{\fm}(R))$.
\end{remark}
The following is known to experts.
\begin{lemma}\label{HSL locali} Let $(R, \fm)$ be an $F$-finite local ring of dimension $d$. Let $\fp \in \mathrm{Spec}(R)$ a prime ideal. Then for all $i \le d$ we have $\mathrm{HSL}(H^{i- \dim R/\fp}_{\fp R_{\fp}}(R_{\fp})) \le \mathrm{HSL}(H^i_{\fm}(R))$. In particular, $\mathrm{HSL}(R_{\fp}) \le \mathrm{HSL}(R)$.
\end{lemma}
\begin{proof} Since $(R, \fm)$ is $F$-finite, it is an image of some regular local ring $(S, \frak n)$ by Gabber \cite{G04}. Suppose that $\dim S = n$. For each $i \le d$, by the definition we have $\mathrm{HLS}(H^i_{\fm}(R))$ is the smallest non-negative integer $e$ such that 
$$\mathrm{Ker}(H^i_{\fm}(R) \xrightarrow{F^e} H^i_{\fm}(F^e_*R)) = \mathrm{Ker}(H^i_{\fm}(R) \xrightarrow{F^{e+1}} H^i_{\fm}(F^{e+1}_*R)).$$
By the local duality theorem we have
$$\mathrm{Coker}(\mathrm{Ext}^{n-i}_S(F^e_*R,S) \xrightarrow{(F^e)^{\vee}} \mathrm{Ext}^{n-i}_S(R,S)) = \mathrm{Coker}(\mathrm{Ext}^{n-i}_S(F^{e+1}_*R,S) \xrightarrow{(F^{e+1})^{\vee}} \mathrm{Ext}^{n-i}_S(R,S)).$$
Let $P$ be the preimage of $\fp$ in $S$. Taking localization we have
$$\mathrm{Coker}(\mathrm{Ext}^{n-i}_{S_P}(F^e_*R_{\fp},S_P) \to \mathrm{Ext}^{n-i}_{S_P}(R_{\fp},S_P)) = \mathrm{Coker}(\mathrm{Ext}^{n-i}_{S_P}(F^{e+1}_*R_{\fp},S_P) \to \mathrm{Ext}^{n-i}_{S_P}(R_{\fp},S_P)).$$
Applying the local duality theorem for $(R_{\fp}, \fp R_{\fp})$ we have 
$$\mathrm{Ker}(H^{i- \dim R/\fp}_{\fp R_{\fp}}(R_{\fp}) \xrightarrow{F^e} H^{i- \dim R/\fp}_{\fp R_{\fp}}(F^e_*R_{\fp})) = \mathrm{Ker}(H^{i- \dim R/\fp}_{\fp R_{\fp}}(R_{\fp}) \xrightarrow{F^{e+1}} H^{i- \dim R/\fp}_{\fp R_{\fp}}(F^{e+1}_*R_{\fp})).$$
Hence $\mathrm{HSL}(H^{i- \dim R/\fp}_{\fp R_{\fp}}(R_{\fp})) \le e$. The proof is complete.
\end{proof}
We are ready to prove the main result of this paper.
\begin{theorem}
 Let $(R,\fm)$ be a local ring of dimention $d$ and of prime characteristic $p>0$, and $t \le d$ an integer such that $H_{\fm}^j(R)/0^F_{H^j_{\fm}(R)}$ has finite length for all $j<t$. Then there exists an integer $C$ such that for any filter regular sequence $x_1,\ldots,x_t$ we have $\mathrm{Fte}(x_1, \ldots, x_t) \leq C$.
\end{theorem}
\begin{proof} By passing to the completion we can assume that $(R, \fm)$ is complete. We next use the $\Gamma$-construction of Hochster and Huneke \cite{HH94} to obtain a faithfully flat extension $R \to R^{\Gamma}$ such that $R^{\Gamma}$ is $F$-finite and $\fm R^{\Gamma}$ is the maximal ideal of $R^{\Gamma}$. Notice that the induced Frobenius actions on $H^j_{\fm}(R)/0^F_{H^j_{\fm}(R)}$ are injecitve for all $j \ge 0$. By \cite[Lemmas 2.9 and 4.3]{EH07} we can choose a sufficiently small $\Gamma$ such that the Frobenius actions on 
$R^{\Gamma} \otimes_R H^j_{\fm}(R)/0^F_{H^j_{\fm}(R)}$ are injecitve for all $j \ge 0$. Thus $0^F_{H^j_{\fm}(R^{\Gamma})} \cong R^{\Gamma} \otimes_R 0^F_{H^j_{\fm}(R)}$. Since $(x_1, \ldots, x_t)^F R^{\Gamma} \subseteq ((x_1, \ldots, x_t)R^{\Gamma})^F$, it is enough to prove the requirement for $R^{\Gamma}$. Therefore we can assume henceforth that $(R,\fm)$ is $F$-finite.

 Let $c$ be the upper bound in Proposition \ref{case dim zero} and $h: = \mathrm{HSL}(R)$. Let $I = (x_1, \ldots, x_t)$, and pick any $a \in I^F \setminus I$. Suppose $\dim R/(I:a) = s$. We will prove $\dim R/(I^{[p^{hs}]}: a^{p^{hs}}) \le 0$ by induction on $s$. There is nothing to do if $s = 0$. Suppose $s >0$, and let $\fp$ be any minimal prime of $R/(I:a)$ with $\dim R/\fp = s$. Then $aR_{\fp} \in (IR_{\fp})^F$ by \cite[Lemma 3.3]{QS17}. By the minimality we have $a + IR_{\fp} \in H^0_{\fp R_{\fp}}(R_{\fp}/IR_{\fp})$. Since $x_1, \ldots, x_t$ becomes a regular sequence in $R_{\fp}$ we have $\mathrm{HSL}_{R_{\fp}} (H^0_{\fp R_{\fp}}(R_{\fp}/IR_{\fp})) \le \mathrm{HSL} (H^t_{\fp R_{\fp}}(R_{\fp})) \le h$ by Remark \ref{regular case} and Lemma \ref{HSL locali}. Therefore $F^h_{R_{\fp}} (a + IR_{\fp}) = 0 \in R_{\fp}/I^{[p^h]}R_{\fp}$, and so $a^{p^h}R_{\fp} \in I^{[p^h]}R_{\fp}$. Hence $R_{\fp}/(I^{[p^h]}: a^{p^h})R_{\fp} = 0$ for all $\fp \in \mathrm{Min}(R/(I:a))$ with $\dim R/\fp = s$. Thus $\dim R/(I^{[p^h]}: a^{p^h}) \le s-1$. The claim now follows from the inductive hypothesis for $I^{[p^h]}$ and $a^{p^h}$. 
  
On the other hand $s \le d-t$, so we always have $\dim R/(I^{[p^{(d-t)h}]}: a^{p^{(d-t)h}}) \le 0$. Moreover,
 if $\dim R/(I: a) \le 0$, then $\bar{a}=a+I\in I^F\cap (I:\fm^{\infty})/I=0^{F_R}_{H_{\fm}^0(R/I)}$ by Lemma \ref{Fro0}. By Proposition \ref{case dim zero} we have $F^c_R(\bar{a})=0$. Thus $a^{p^c} \in I^{[p^c]}$. Putting all together we have $a^{p^{(d-t)h + c}} \in I^{[p^{(d-t)h + c}]}$ for all $a \in I^F$, so $\mathrm{Fte}(I) \le (d-t)h + c$. The proof is complete.
\end{proof}
\begin{acknowledgement} The authors are grateful to the referee for carefully reading of the paper and valuable suggestions and comments. The first author was funded by Vingroup Joint Stock Company and supported by the Domestic Master/ PhD Scholarship Programme of Vingroup Innovation Foundation (VINIF), Vingroup Big Data Institute (VINBIGDATA).
\end{acknowledgement}

\end{document}